\documentclass[10pt]{article}
\usepackage{amsfonts}

\usepackage{amsmath}
\usepackage{graphicx}

\newtheorem{theorem}{Theorem}[section]

\newtheorem{corollary}[theorem]{Corollary}

\newtheorem{lemma}[theorem]{Lemma}

\newtheorem{problem}[theorem]{Problem}
\newtheorem{proposition}[theorem]{Proposition}
\newtheorem{remark}[theorem]{Remark}

\newenvironment{proof}[1][Proof]{\textbf{#1.} }{\ \rule{0.5em}{0.5em}}

\def\RR{\mathbb{R}}
\def\EE{\mathbb{E}}

\def\cF{{\cal F}}

\def\de{{\delta}}
\def\la{{\lambda}}
\def\si{{\sigma}}

\def\Om{{\Omega}}
\def\om{{\omega}}

\def\ga{{\gamma}}
\def\de{{\delta}}

\def\si{{\sigma}}

\def\la{{\lambda}}

\def\vare{{\varepsilon}}

\begin{document}

\title{ A singular SDE driven by additive
fractional Brownian motion with Hurst parameter $H<\frac12$}
\author{Xiaoming Song and Alexander Tortoriello\\
\small Department of Mathematics\\
\small Drexel University\\
\small Philadelphia, PA 19104, USA\\
\small xs73@drexel.edu, at3625@drexel.edu}
\date{\today}
\maketitle
\begin{abstract} 
In this article we study a class of singular  stochastic differential equations driven by fractional Brownian motion with Hurst parameter $H<\frac12$. The solution is constructed as the limit of a family of approximating processes, and its trajectory properties  are  investigated.
\end{abstract}

\section{Introduction}
We are interested in the study of a class of singular stochastic differential equations of the following form
\begin{equation}\label{intro-1}
X_t^H=X_0+a\int_0^t\frac{s^{2H-1}}{X_s^H}ds-b\int_0^t X_sds+\sigma B_t^H,
\end{equation}
where $X_0>0$, $a>0$, $b\geq 0$, $\si>0$ are given constants, and $B^H=\{B^H, t\geq 0\}$ is a fractional Brownian motion with Hurst parameter $H<\frac12$. The function $f(t,x)=\frac{at^{2H-1}}{x}-bx$ is called the drift of equation \eqref{intro-1}.

 The existence of weak and strong solutions for this type of stochastic differential equations has been studied under different assumptions on the drift $f(t,x)$  in both cases $H>\frac12$ and $H\leq \frac12$. Using Girsanov's theorem for fractional Brownian motion, Nualart and Ouknine in \cite{MR1934157} proved the existence and uniqueness of a strong solution under the assumption that $f(t,x)$ satisfies the linear growth condition $|f(t,x)|\leq C(1+|x|)$ if $H\leq\frac12$, and that $f(t,x)$ is H\"older continuous of order $\alpha>1-\frac{1}{2H}$ in $x$ and of order $\gamma>H-\frac12$ in $t$ if $H>\frac12$. Boufoussi and Ouknine \cite{MR2042751} extended the result in the case $H>\frac12$ to a drift given as the sum of  a H\"older continuous function and a nondecreasing function with left (or right) continuity. By means of Girsanov's theorem, Mishura and Nualart \cite{MR2125162} obtained the existence of a weak solution   for $H\in \left(\frac12, \frac{1+\sqrt{5}}{4}\right)$ assuming that the drift  $f(t,x)=f(x)$ has finitely many jump-discontinuities and is piecewise H\"oder continuous of order $\alpha\in \left(1-\frac{1}{2H}, 1\right)$ on each  interval where the drift is continuous.

The goal of this paper is to investigate the stochastic differential equation \eqref{intro-1} for $H<\frac12$, and in this case the drift $f(t,x)=\frac{at^{2H-1}}{x}-bx$ has singularities at $t=0$ and $x=0$. Although the drift contains a linear term in $x$,  it does not satisfy the linear growth condition in $x$. So, the result in Nualart and Ouknine \cite{MR1934157} cannot be applied directly to our work. For this type of singular drive, in the case of $H>\frac12$ and $b=0$, Hu, Nualart and Song in \cite{MR2458016} used the techniques of fractional calculus for the Riemann-Stieltjes integral to obtain a unique positive and continuous solution with moments of all orders. Also, by means of Malliavin calculus, the authors  proved that the law of the solution $X_t$ has a density with respect to the Lebesgue measure for any $t>0$. However, for the case $H<\frac12$, the method of fractional calculus does not work anymore because the order of H\"older continuity of $B^H$ is too small.  Inspired by the work of Mishura and Yurchenko-Tytarenko  \cite{MR3935425} in which they considered the drift function of the form $f(t,x)=\frac{k}{x}-bx$, we construct a family of approximating processes $X^\vare=\{X^\vare_t, t\geq0\}$, where $X^\vare$ satisfies a stochastic differential equation whose drift has a linear growth in $x$. As $\vare\to 0$, $\{X^\vare\}$ converges to a process $X$ which is defined as the solution of \eqref{intro-1}. We prove that the trajectories of the solution $X$ are  nonnegative and almost everywhere continuous on $[0,\infty)$ almost surely. Unlike the drift in \cite{MR3935425},  the singular function $t^{2H-1}$ appears in our drift $f(t,x)$ and it should be also approximated by an appropriate function. This singular function requires the development of new techniques to prove the continuity of the solution’s trajectories .

This type of singular stochastic differential equations is partially motivated by the equation satisfied by the $d$-dimensional fractional Bessel process $R_t=\sqrt{\sum_{i=1}^d(B_t^{H,i})^2}$, $d\geq 2$ (see Essaky and Nualart \cite{MR3385597} for $H<\frac12$ and Guerra and Nualart \cite{MR2105371} for $H>\frac12$): 
\[
R_t=H(d-1)\int_0^t\frac{s^{2H-1}}{R_s}ds+Y_t,
\]
where  $Y_t=\int_0^t\sum_{i=1}^d\frac{B_s^{H,i}}{R_s}\delta B_s^{H,i}$ is a divergence integral, and the process $Y=\{Y_t, t\geq 0\}$ is $H$-self-similar. It has been proved that the process $Y=\{Y_t, t\geq 0\}$ is not a one-dimensional fractional Brownian motion by Hu and Nualart \cite{MR2137449} in the case $H\neq\frac12$. 

In  the case of $H=\frac12$, $B^{\frac12}$ is the standard Brownian motion. If $W=\{(W_t^1, \dots, W_t^d), t\geq 0 \}$, $d\geq 2$, is the $d$-dimensional Brownian motion, then the $d$-dimensional Bessel process $R_t=\sqrt{\sum_{i=1}^d(W_t^{i})^2}$, satisfies the following differential equation:
\[
R_t=\frac{d-1}{2}\int_0^t \frac{1}{R_s}ds+\sum_{i=1}^d\int_0^t\frac{W^i_s}{R_s}dW_s^i,
\]
where the integral is the classical It\^o's integral and the process $Y=\{Y_t=\sum_{i=1}^d\frac{W^i_s}{R_s}dW_s^i, t
\geq 0\}$ is a one-dimensional Brownian motion (see Karatzas and Shreve \cite{MR1121940} for more details). Notice that the initial condition is zero in the stochastic differential equations satisfied by Bessel processes.

The rest of the paper is organized as follows. In Section 2, we introduce a family  of approximating processes $\{X^\vare\}$ and use a comparison criterion to prove the monotonicity of the processes in $\vare$. After a boundedness result is obtained, the limiting process can be derived and will be  considered as a solution of the stochastic differential equation. Section 3 is devoted to study the trajectory properties of the solution, such as nonnegativity, continuity, Riemann integrability.

\setcounter{equation}{0}
\section{Approximation processes}

Let $B^H=\{B^H_t, \, t\geq 0\}$ be a fractional Brownian motion with Hurst parameter $H\in (0,1)$, defined in a complete probability space $(\Omega, \cF, \mathbb{P})$. Namely, $B^H$ is a mean zero Gaussian
process with covariance
\begin{equation*}
\mathbb{E}(B_{t}^{H}B_{s}^{H})=R_{H}(s,t)=\frac{1}{2}\left(
t^{2H}+s^{2H}-|t-s|^{2H}\right) \,.  \label{cov}
\end{equation*}%

It implies from the above covariance function that for any $s, t\geq 0$, we have 
\begin{equation}\label{fbm-inc}
\EE(|B_t^H-B_s^H|^2)=|t-s|^{2H}.
\end{equation}

Thus, the fractional Brownian motion $B^H$ has stationary increment, and $B^H_t-B^H_s$ has normal distribution $\mathcal{N}(0, |t-s|^{2H})$. In the case $H=\frac12$, $B^{\frac12}$ is the standard Brownian motion.

By Kolmogorov's continuity criterion and Garsia-Rodemich-Rumsey Lemma (see \cite{MR267632} or Section A.3 in \cite{MR2200233}), we deduce that $B^H$ has a version with continuous trajectories satisfying the following property.
\begin{itemize}
\item[(a)] For any $T>0$
\[
\sup\limits_{s\in[0,T]}|B^H_s(\om)|<+\infty, \ \forall \om\in \Om.
\]
\item[(b)] For any $T>0$ and for any $\ga\in (0,H)$ , there exists a nonnegative random variable $C_{\ga,T}$ such that $\EE(C_{\ga,T}^p)<+\infty$ for all $p\geq 1$,
\begin{equation}\label{fbm-hold}
|B_t^H(\om)-B_s^H(\om)|\leq C_{\ga,T}(\om)|t-s|^{H-\ga}, \ \forall s, t\in [0,T], \ \om\in\Omega.
\end{equation}
\end{itemize}

\bigskip
In this section, let us fix $T>0$. Consider
\begin{equation}\label{sde-1}
Y_t=y_0+\int_0^tb(s,Y_s)ds+\si B^H_t, \ t\in[0,T],
\end{equation}
where $y_0\in\RR$, $b:[0,T]\times \mathbb{R}\to \mathbb{R}$ is measurable.

\begin{itemize}
\item[(i)]
For $H\leq\frac12$, if there is a constant $C>0$ such that
\[
|b(t,y)|\leq C(1+|y|),\ \forall  t\in [0,T],\ y\in\RR,
\]
 then it follows from Theorem 8 in \cite{MR1934157} that equation \eqref{sde-1} has a unique strong solution and the solution has continuous trajectories.

\item[(ii)] For $H>\frac12$, if there exist some constants $L>0$, and some function $b_0\in L^\rho([0,T])$ for some $\rho\geq 2$ such that
\[
|b(t,y)|\leq L|y|+b_0(t),\ \forall t\in[0,T],  y\in \RR,
\]
and if for any $N>0$ there exists a constant $C_N>0$ such that 
\[
|b(t,x)-b(t,y)|\leq C_N|x-y|, \ \forall  |x|, |y|\leq N, \ t\in[0,T],
\]

then Theorem 2.1 in \cite{MR1893308} guarantees the existence and uniqueness of the solution of \eqref{sde-1}.
\end{itemize}

 For any $\vare>0$, consider the stochastic process $X^\vare=\{X^\vare_t,\, t\in [0,T] \}$ satisfying the following equation
\begin{equation}\label{sde-2}
X^\vare_t=X_0+a\int_0^t\frac{(s+\vare)^{2H-1}}{X_s^\vare\boldsymbol{1}_{\{X^\vare_s>0\}}+\vare}ds-b\int_0^tX^\vare_sds+\si B^H_t,\quad t\in[0,T].
\end{equation}

Notice that the function $b^\vare(t,x)=\frac{a(t+\vare)^{2H-1}}{x\boldsymbol{1}_{\{x>0\}}+\vare}-bx$ satisfies both (i) and (ii). Hence, equation \eqref{sde-2} has a unique strong solution on $[0,T]$.\\

Let us first present the following comparison criterion.
\begin{proposition}\label{prop-comp}
Let $b(t)$ be a continuous function on $[0,\infty)$. Assume that continuous functions $x^i=\{x^i_t,\, t\geq 0\}$, $i=1, 2$, satisfy the following equations
\begin{equation*}
x^i_t=x_0+\int_0^tg_i(s)f_i(x^i_s)ds+\int_0^t h_i(x^i_s)ds+b(t), \ t\geq 0,
\end{equation*}
where $x_0\in\mathbb{R}$ is a constant, and $g_i, f_i, h_i$, $i=1, 2$, are continuous functions such that 
\begin{itemize}
\item[(a)] $0<g_1(t)<g_2(t)$, $\forall t\geq 0$;
\item[(b)] $0< f_1(x)<f_2(x)$, $\forall x\in\RR$;
\item[(c)] $h_1(x)\leq h_2(x)$, $\forall x\in\RR$.
\end{itemize}
Then, for all $t> 0$, we have $x^1_t<x^2_t$.
\end{proposition}
\begin{proof}
 Denote $\de(t)=x^2_t-x^1_t$. Then, we have
\[
\de(t)=\int_0^t(g_2(s)f_2(x^2_s)-g_1(s)f_1(x^1_s))ds+\int_0^t(h_2(x^2_s)-h_1(x^1_s))ds.
\]
Note that the function $\de(\cdot)$ is differentiable and $\de(0)=0$. Note also that
\begin{align*}
\de'(0)&=\left(g_2(0)f_2(x_0)-g_1(0)f_1(x_0)\right)+\left(h_2(x_0)-h_1(x_0)\right)\\
&=g_2(0)(f_2(x_0)-f_1(x_0))+(g_2(0)-g_1(0))f_1(x_0)+\left(h_2(x_0)-h_1(x_0)\right)>0.
\end{align*}
So, there exists some extended real number $t_0\in (0,+\infty]$ such that 
\[
\de(t)>0, \ \forall t\in (0, t_0).
\]
We can easily see that 
\begin{equation}\label{t-0}
t_0=\sup\{t>0: \text{ for all } s\in (0, t), \de(s)>0\}.
\end{equation}
Assume that $t_0<+\infty$.  Since the function $\de$ is differentiable and hence continuous, we must have $\de(t_0)=0$, that is, $x^1_{t_0}=x^2_{t_0}=x$ for some $x\in \mathbb{R}$. Notice that 
\begin{align*}
\de'(t_0)&=\left(g_2(t_0)f_2(x)-g_1(t_0)f_1(x)\right)+\left(h_2(x)-h_1(x)\right)\\
&=g_2(t_0)(f_2(x)-f_1(x))+(g_2(t_0)-g_1(t_0))f_1(x)+\left(h_2(x)-h_1(x)\right)>0.
\end{align*}
Thus, there exists some $\vare_0>0$ such that on $(t_0-\vare_0, t_0)$ we have $\de(t)<0$, which contradicts the definition of $t_0$ in \eqref{t-0}. So, we must have $t_0=+\infty$. Therefore, $\de(t)>0$ for all $t>0$, that is, $x^1_t<x^2_t$ for all $t>0$.
\end{proof}\\

\bigskip

In the sequel, we will fix $H\in (0,\frac12)$. For any $0<\vare_2<\vare_1$, let $b(t)=\si B_t^H$ and 
\[
g_i(t)=(t+\vare_i)^{2H-1}, \ f_i(x)=\frac{a}{x\boldsymbol{1}_{\{x>0\}}+\vare_i},  \text{ and } h_i(x)=-bx, \  i=1,2.
\] Note that these functions satisfy 
the monotonicity conditions in Proposition \ref{prop-comp}. Applying Proposition \ref{prop-comp}, we have, for all $\om\in \Om$ and all $t\in (0,T]$, 
\[
X^{\vare_2}_t(\om)>X^{\vare_1}_t(\om).
\]

Since we consider a version of fractional Brownian motion satisfying (a) and (b) for all $\om\in \Om$, we will fix $\om\in \Om$ and omit it in all the proofs of the results whenever the continuity or H\"older continuity of fractional Brownian motion is concerned in this paper.

Next, we will show that the collection $\{X^{\vare}\}$ has a uniform upper bound and hence we can define
\[
X_t=\lim\limits_{\vare\to 0}X_t^\vare<+\infty,\ \forall t\in[0,T].
\]
\begin{theorem}\label{thm-2-2}
There exists a positive constants $C=C(T, H,X_0, a)$ such that for all $\vare>0$ and all $t\in[0,T]$ we have
\begin{equation}\label{upper}
X_t^\vare\leq C+2\si\sup\limits_{s\in[0,T]}|B^H_s|.
\end{equation}
\end{theorem}
\begin{proof}
Define
\[
\tau_1^\vare=\sup\left\{t\in[0,T]:  X^\vare_s\geq \frac{X_0}{2}, \ \forall s\in [0,t]\right\}.
\]
Since $X^\vare_0=X_0$ and $X^\vare$ is continuous, we have $0<\tau_1^\vare\leq T$ and $X^\vare_t\geq \frac{X_0}{2}$ for all $t\in [0,\tau_1^\vare]$.

For $t\in [0,\tau_1^\vare]$, we obtain
\begin{align}\label{upper-1}
X^\vare_t&=X_0+a\int_0^t\frac{(s+\vare)^{2H-1}}{X_s^\vare\boldsymbol{1}_{\{X^\vare_s>0\}}+\vare}ds-b\int_0^tX^\vare_sds+\si B^H_t\notag\\
&\leq X_0+\frac{2a}{X_0}\int_0^t (s+\vare)^{2H-1}ds+\si\sup\limits_{0\leq s\leq T}|B_s^H|\notag\\
&=X_0+\frac{a}{HX_0}[(t+\vare)^{2H}-\vare^{2H}]+\si\sup\limits_{0\leq s\leq T}|B_s^H|\notag\\
&\leq X_0+\frac{aT^{2H}}{HX_0}+\si\sup\limits_{0\leq s\leq T}|B_s^H|,
\end{align}
since for $0<H<\frac12$ the inequality $|u^{2H}-v^{2H}|\leq |u-v|^{2H}$ holds for any $u,v\geq 0$. \\

\textbf{Case 1:} $\tau_1^\vare =T$. In this case, the desired result \eqref{upper} follows from \eqref{upper-1} with $C=X_0+\frac{aT^{2H}}{HX_0}$.\\

\textbf{Case 2:} $\tau_1^\vare<T$. For any $t\in(\tau_1^\vare,T]$.
by the definition of $\tau_1^\vare$, there must exist some $s\in (\tau_1^\vare,t)$ such that $X^\vare_s<\frac{X_0}{2}$. 

If  $X^\vare_t\leq \frac{X_0}{2}$,  we obtain \eqref{upper} with $C=\frac{X_0}{2}$.

If $X^\vare_t> \frac{X_0}{2}$, define
\[
\tau^\vare_2=\sup\left\{s\in (\tau_1^\vare, t): X_s^\vare\leq \frac{X_0}{2} \right\}.
\]
Since $X^\vare$ is continuous, we know that $\tau_2^\vare\in (\tau_1^\vare, t)$, $X^\vare_{\tau_2^\vare}=\frac{X_0}{2}$, and $X_s^\vare>\frac{X_0}{2}$ for all $s\in (\tau_2^\vare, t]$. Thus, we obtain
\begin{align}\label{upper-2}
X_t^\vare&=X_t^\vare-X_{\tau^\vare_2}^\vare+X_{\tau^\vare_2}^\vare\notag\\
&=a\int_{\tau^\vare_2}^t\frac{(s+\vare)^{2H-1}}{X_s^\vare\boldsymbol{1}_{\{X_s^\vare>0\}}+\vare}ds-b\int_{\tau^\vare_2}^tX_s^\vare ds+\sigma(B^H_t-B^H_{\tau_2^\vare})+\frac{X_0}{2}\notag\\
&\leq\frac{2a}{X_0}\int_{\tau^\vare_2}^t(s+\vare)^{2H-1}ds+2\si \sup\limits_{0\leq s\leq T}|B_s^H|+\frac{X_0}{2}\notag\\
&=\frac{a}{HX_0}[(t+\vare)^{2H}-(\tau_2^\vare+\vare)^{2H}]+2\si \sup\limits_{0\leq s\leq T}|B_s^H|+\frac{X_0}{2}\notag\\
&\leq \frac{a}{HX_0}(t-\tau_2^\vare)^{2H}+2\si \sup\limits_{0\leq s\leq T}|B_s^H|+\frac{X_0}{2}\notag\\
&\leq \frac{aT^{2H}}{HX_0}+2\si \sup\limits_{0\leq s\leq T}|B_s^H|+\frac{X_0}{2}.
\end{align}
So, in this case the desired result \eqref{upper} follows from \eqref{upper-2} with $C=\frac{X_0}{2}+\frac{aT^{2H}}{HX_0}$.

Therefore, for any $t\in [0,T]$, combining all the cases and taking $C=X_0+\frac{aT^{2H}}{HX_0}$, we can obtain \eqref{upper}.
\end{proof}

\begin{remark}\label{rem-1}
It implies from the monotonicity of the family $\{X^\vare\}_{\vare>0}$ and the upper bound \eqref{upper} that  the limit $\lim\limits_{\vare\to 0}X_t^\vare$ exists for each $t\in [0,T]$ and we denote 
\begin{equation}\label{limit}
X_t=\lim\limits_{\vare\to 0}X_t^\vare, \ \forall t\in [0,T].
\end{equation}
From the definition of $X$ and \eqref{upper} we obtain
\begin{equation}\label{limit-bound}
X_t^\vare< X_t\leq C+2\si \sup\limits_{s\in[0,T]}|B^H_s|, \ \forall t\in [0,T].
\end{equation}
for all $\vare>0$. The process $X$ as the limit of continuous processes $X^\vare$ is progressively measurable on $[0,T]\times\Om$. By the monotone convergence theorem, the process $X$ is Lebesgue integrable on arbitrary interval $[0,t]\subseteq[0,T]$, and 
\begin{equation}\label{limit-int}
\int_0^tX_sds=\lim\limits_{\vare\to0}\int_0^t X_s^\vare ds<\infty, \ \forall t\in [0,T].
\end{equation}
In the sequel, the integral of $X$ will be considered as the Lebesgue integral at first and later we will prove that $X$ is Riemann integrable on any interval $[0,t]$.
\end{remark}
\begin{remark}
Since $T>0$ is arbitrary, for any $t\geq 0$ we can choose large enough $T>0$ such that $t\leq  T$ and consider $X^\vare=\{X^\vare_s, \ s\in [0,T]\}$ for any $\vare>0$. Then, we can define 
\begin{equation}\label{limit-1}
X_t=\lim\limits_{\vare\to 0}X^\vare_t.
\end{equation}
Hence, we can obtain a process $X=\{X_t, t\geq 0\}$.
\end{remark}
\setcounter{equation}{0}
\section{Properties of the approximations and the limit process}

Let us first prove the following property for the approximation properties.
\begin{proposition}\label{prop-Leb}
Let $T>0$ be fixed and let $\la$ denote the Lebesgue measure on $[0,T]$. Then,
\begin{equation}\label{limit-Leb}
\lim\limits_{\vare\to 0}\la\left(\{t\in[0,T]: X^\vare_t\leq 0\}\right)=0.
\end{equation}
\end{proposition}
\begin{proof}
 For any $t\in [0,T]$, it follows from \eqref{sde-2} that 
\begin{equation}\label{sde-3}
X^\vare_t-X_0+b\int_0^tX^\vare_sds-\si B_t^H=a\int_0^t\frac{(s+\vare)^{2H-1}}{X_s^\vare\boldsymbol{1}_{\{X^\vare_s>0\}}+\vare}ds.
\end{equation}
Equations \eqref{limit} and \eqref{limit-int} yield
\begin{equation}\label{limit-LHS}
\lim\limits_{\vare\to 0}\left(X^\vare_t-X_0+b\int_0^tX^\vare_sds-\si B_t^H\right)=X_t-X_0+b\int_0^tX_sds-\si B_t^H.
\end{equation}
 Hence, the limit of the right-hand side of \eqref{sde-3} exists and 
 \begin{equation}\label{limit-RHS}
 \lim\limits_{\vare\to 0} a\int_0^t\frac{(s+\vare)^{2H-1}}{X_s^\vare\boldsymbol{1}_{\{X^\vare_s>0\}}+\vare}ds\in [0, \infty), \ \forall t\in[0,T].
 \end{equation}
 
 Assume that \eqref{limit-Leb} does not hold. Then, there exist a sequence $\{\vare_n>0\}$ and some $\delta_0>0$ such that $\{\vare_n\}$ decreases to $0$ and 
 \[
 \la\left(\{t\in[0,T]: X^{\vare_n}_t\leq 0\}\right)\geq \delta_0>0.
 \] 
 Then,
 \begin{align}\label{ineq-contra}
 &\int_0^T\frac{(s+\vare_n)^{2H-1}}{X_s^{\vare_n}\boldsymbol{1}_{\{X^{\vare_n}_s>0\}}+\vare_n}ds\notag\\
 =&\int_{\left\{s\in[0,T]: X^{\vare_n}_s\leq 0\right\}}\frac{(s+\vare_n)^{2H-1}}{X_s^{\vare_n}\boldsymbol{1}_{\{X^{\vare_n}_s>0\}}+\vare_n}ds+\int_{\left\{s\in[0,T]: X^{\vare_n}_s> 0\right\}}\frac{(s+\vare_n)^{2H-1}}{X_s^{\vare_n}\boldsymbol{1}_{\{X^{\vare_n}_s>0\}}+\vare_n}ds\notag\\
 \geq & \int_{\{s\in[0,T]: X^{\vare_n}_s\leq 0\}}\frac{(s+\vare_n)^{2H-1}}{X_s^{\vare_n}\boldsymbol{1}_{\{X^{\vare_n}_s>0\}}+\vare_n}ds\notag\\
 =&\int_{\{s\in[0,T]: X^{\vare_n}_s\leq 0\}}\frac{(s+\vare_n)^{2H-1}}{\vare_n}ds\notag\\
 \geq &\frac{(T+\vare_n)^{2H-1}\la\left(\{t\in[0,T]: X^{\vare_n}_t\leq 0\}\right)}{\vare_n}\notag\\
 \geq& \frac{(T+\vare_n)^{2H-1}\delta_0}{\vare_n}.
 \end{align}
 Noticing that $\lim\limits_{n\to\infty}\frac{\delta_0(T+\vare_n)^{2H-1}}{\vare_n}=+\infty$, we obtain from \eqref{ineq-contra} that 
 \[
 \lim\limits_{n\to\infty} a\int_0^T\frac{(s+\vare_n)^{2H-1}}{X_s^{\vare_n}\boldsymbol{1}_{\{X^{\vare_n}_s>0\}}+\vare_n}ds=+\infty,
 \]
 which contradicts \eqref{limit-RHS}. Therefore, the limit \eqref{limit-Leb} must hold true.
 \end{proof}
\begin{corollary}\label{coro-1}
The process $X$ is positive almost everywhere on $[0,\infty)$.
Moreover, we have
\[
\{t\geq 0: X_t\leq 0\}=\{t\geq 0: X_t=0\}.
\]
So, the process $X$ is nonnegative on $[0,\infty)$.

\end{corollary}
\begin{proof}
For any $T>0$, by the monotonicity of $\{X^\vare\}$ we have
\[
\left\{t\in[0,T]: X^\vare_t\leq 0\right\}\downarrow \left\{t\in[0,T]: X_t\leq 0\right\},
\]
as $\vare\downarrow 0$.

Then, it implies from Proposition \ref{prop-Leb} and the monotone convergence theorem that
\[
\la\left(\{t\in[0,T]: X_t\leq 0\}\right)=\lim\limits_{\vare\to0}\la\left(\{t\in[0,T]: X^\vare_t\leq 0\}\right)=0.
\]
Namely, we have $X_t>0$ almost everywhere on $[0,T]$, a.s. Since $T$ is arbitrary, we obtain $X_t>0$ almost everywhere on $[0,\infty)$.

Now, let us prove $\{t\geq 0: X_t\leq 0\}=\{t\geq 0: X_t=0\}$. Fix $\om\in \Om$ and we will omit $\om$ in the following context. 

It is obvious that $\{t\geq 0: X_t=0\}\subseteq\{t\geq 0: X_t\leq 0\}$.

Next, we shall prove $\{t\geq 0: X_t\leq 0\}\subseteq\{t\geq 0: X_t=0\}$. Assume that there is some $\tau>0$ such that $X_\tau\leq 0$. In the following we consider $0<\vare\leq 1$. Then, we have
\[
X^\vare_\tau<X_\tau\leq 0, \ \forall \vare\in(0,1].
\]
Now, define
\begin{align*}
\tau_{-}^\vare&=\sup\{t\in(0,\tau): X^\vare_t>0\},\\
\tau_{+}^\vare&=\inf\{t\in(\tau, \infty): X^\vare_t>0\}.
\end{align*}
It implies from the continuity of $X^\vare$ that $0<\tau_{-}^\vare<\tau<\tau_{+}^\vare<\infty$, $X^\vare_{\tau_{-}^\vare}=0$ and $X^\vare_{\tau_{+}^\vare}=0$ for all $0<\vare\leq 1$. By the monotonicity of $X^\vare$ we obtain, for any $0<\vare_2<\vare_1\leq 1$,
\[
\tau_{-}^{1}\leq\tau_{-}^{\vare_1}\leq \tau_{-}^{\vare_2} <\tau_{+}^{\vare_2}\leq \tau_{+}^{\vare_2} \leq \tau_{+}^1.
\]
It is easy to see that 
\[
X^\vare_t< 0, \ \forall t\in (\tau_{-}^{\vare}, \tau_{+}^{\vare}).
\]
It deduces from \eqref{fbm-hold} that there exist $\beta\in (0,H)$ and $C>0$ such that
\[
|B^H_s-B^H_t|\leq C|t-s|^\beta, \ \forall s, t\in [0,\tau^1_{+}].
\]
Then, we have, for $t\in (\tau_{-}^\vare, \tau_{+}^\vare)$,
\begin{align*}
0>X^\vare_t&=X^\vare_t-X^\vare_{\tau_{-}^\vare}\\
&=a\int_{\tau_{-}^\vare}^t\frac{(s+\vare)^{2H-1}}{X^\vare_s\boldsymbol{1}_{\{X_s^\vare>0\}}+\vare}ds-b\int_{\tau_{-}^\vare}^t X^\vare_sds+\si (B^H_t-B^H_{\tau_{-}^\vare})\\
&>\frac{a(\tau_{+}^1+\vare)^{2H-1}}{\vare}(t-\tau_{-}^\vare)-C\si (t-\tau_{-}^\vare)^\beta.
\end{align*}
Thus, in order to prove $X_\tau=0$, it suffices to prove that 
\[
h(\vare)=\min\limits_{t\in [\tau_{-}^\vare, \tau_{+}^\vare]}\frac{a(\tau_{+}^1+\vare)^{2H-1}}{\vare}(t-\tau_{-}^\vare)-C\si (t-\tau_{-}^\vare)^\beta\to 0,
\]
as $\vare\to 0$. Put
\[
g_\vare(t)=\frac{a(\tau_{+}^1+\vare)^{2H-1}}{\vare}(t-\tau_{-}^\vare)-C\si (t-\tau_{-}^\vare)^\beta,\ t\geq \tau_{-}^\vare.
\]
Then,
\begin{align*}
g_\vare(\tau_{-}^\vare)&=0,\\
g'_\vare(t)&=\frac{a(\tau_{+}^1+\vare)^{2H-1}}{\vare}-C\si\beta  (t-\tau_{-}^\vare)^{\beta-1},\\
g''_\vare(t)&=-C\si\beta (\beta-1)(t-\tau_{-}^\vare)^{\beta-2}>0, \ \forall t>\tau_{-}^\vare.
\end{align*}
Setting $g'_\vare(t)=0$, we can obtain the solution
\[
t^\ast=\tau_{-}^\vare+C_0(\vare)\vare^{\frac{1}{1-\beta}},
\]
where $C_0(\vare)=\left(\frac{C\si\beta}{a(\tau_{+}^1+\vare)^{2H-1}}\right)^{\frac{1}{1-\beta}}$.
Noting that $g''_\vare(t^\ast)>0$, by the second derivative test, $g_\vare$ attains its minimum at $t^\ast$ and we have
\begin{align}\label{bound-h}
0\geq h(\vare)\geq g_\vare(t^\ast)&=a(\tau_{+}^1+\vare)^{2H-1}C_0(\vare)\vare^{\frac{\beta}{1-\beta}}-C\si C_0(\vare)^\beta\vare^{\frac{\beta}{1-\beta}}\notag\\
&=\left[a(\tau_{+}^1+\vare)^{2H-1}C_0(\vare)-C\si C_0(\vare)^\beta\right]\vare^{\frac{\beta}{1-\beta}}.
\end{align}
Notice that 
\begin{align*}
&\lim\limits_{\vare\to 0}\left[a(\tau_{+}^1+\vare)^{2H-1}C_0(\vare)-C\si C_0(\vare)^\beta\right]\\
=&a^{-\frac{\beta}{1-\beta}}\left(\tau_{+}^1\right)^{\frac{\beta(1-2H)}{1-\beta}}(C\si\beta)^{\frac{1}{1-\beta}}-C\si \left(\frac{C\si\beta}{a(\tau_{+}^1)^{2H-1}}\right)^{\frac{\beta}{1-\beta}},
\end{align*}
and 
\[
\lim\limits_{\vare\to 0}\vare^{\frac{\beta}{1-\beta}}=0.
\]
Thus, we can obtain $\lim\limits_{\vare\to 0}g_\vare(t^\ast)=0$. Together with \eqref{bound-h}, we have $\lim\limits_{\vare\to 0} h(\vare)=0$. Hence, $\{t\geq 0: X_t\leq 0\}\subseteq\{t\geq 0: X_t=0\}$.

So, we get $\{t\geq 0: X_t\leq 0\}=\{t\geq 0: X_t=0\}$.

Therefore, the process $X$ is nonnegative on $[0,\infty)$.\end{proof}
\begin{corollary}\label{coro-2}
For any $t>0$, we have
\[
\int_0^t\frac{s^{2H-1}}{X_s}ds\in (0, \infty).
\]
\end{corollary}
\begin{proof}
Noting that $ \frac{(s+\vare)^{2H-1}}{X^\vare_s\boldsymbol{1}_{\{X^\vare_s>0\}}+\vare}>0$ for all $\vare>0$ and all $s\in[0,t]$, by Fatou's lemma, Corollary \ref{coro-1}, and \eqref{sde-3}-\eqref{limit-RHS}, we obtain
\begin{align}\label{ineq-X}
0< a\int_0^t\frac{s^{2H-1}}{X_s}ds
=&a\int_0^t\liminf\limits_{\vare\to 0}\frac{(s+\vare)^{2H-1}}{X^\vare_s\boldsymbol{1}_{\{X^\vare_s>0\}}+\vare}ds\notag\\
\leq& \lim\limits_{\vare\to 0} a\int_0^t\frac{(s+\vare)^{2H-1}}{X_s^\vare\boldsymbol{1}_{\{X^\vare_s>0\}}+\vare}ds\notag\\
=&\lim\limits_{\vare\to 0}\left(X^\vare_t-X_0+b\int_0^tX^\vare_sds-\si B_t^H\right)\notag\\
=&X_t-X_0+b\int_0^tX_sds-\si B_t^H<\infty.
\end{align}
The proof is completed.
\end{proof}

\begin{corollary}
There exists a nonnegative process $L^H=\{L_t^H, t\geq 0\}$ such that
\begin{equation}\label{eq-X}
X_t=X_0+a\int_0^t\frac{s^{2H-1}}{X_s}dsds-b\int_0^tX_sds+\si B_t^H+L_t^H, \ \forall t\geq 0.
\end{equation}
\end{corollary}
\begin{proof}
The proof follows immediately from \eqref{ineq-X}.
\end{proof}

\bigskip
Before we discuss the trajectory properties of the limit process $X$, let us first provide a result about the regularity property of the approximations $X^\vare$ with respect to $\vare$.
\begin{theorem}\label{thm-3-5}
For any $\vare^\ast>0$ and for all $t\geq 0$, we have
\[
\lim\limits_{\vare\to\vare^\ast}X^\vare_t=X^{\vare^\ast}_t.
\]
\end{theorem}
 \begin{proof}
 For all $t\geq 0$, by the monotonicity of $\{X^\vare\}$,  we denote
 \[
 X^{+}_t=\lim\limits_{\vare\downarrow\vare^\ast}X^\vare_t, \ \text{ and }\ X^{-}_t=\lim\limits_{\vare\uparrow\vare^\ast}X^\vare_t.
 \]
It is clear that 
\begin{align*}
&X^{2\vare^\ast}_t\leq X^{+}_t\leq X^{\vare^\ast}_t\leq X^{-}_t\leq X^{\frac{\vare^\ast}{2}}_t,\\
 &X^\vare_t\uparrow X^+_t, \ \text{ as }\vare\downarrow\vare^\ast,
 \text{and }\\ &X^\vare_t\downarrow X^-_t,\ \text{ as }\vare\uparrow\vare^\ast.
 \end{align*} 

It follows from the monotonicity of $\{X^\vare\}$, integrability of $\{X^\vare\}$ and the monotone convergence theorem that, for any $t>0$,
\begin{equation}\label{plus-1}
\int_0^t X^{+}_sds=\lim\limits_{\vare\downarrow\vare^\ast}\int_0^tX^\vare_sds
\end{equation}
and 
\begin{equation}\label{minus-1}
\int_0^t X^{-}_sds=\lim\limits_{\vare\uparrow\vare^\ast}\int_0^tX^\vare_sds.
\end{equation}
Notice that
\[
X^\vare_t\boldsymbol{1}_{\{X^\vare_t>0\}}\uparrow X^+_t\boldsymbol{1}_{\{X^+_t>0\}},\ \text{ as }\vare\downarrow\vare^\ast,\\
\]
and
\[
X^\vare_t\boldsymbol{1}_{\{X^\vare_t>0\}}=X^\vare_t\boldsymbol{1}_{\{X^\vare_t\geq0\}}\downarrow X^-_t\boldsymbol{1}_{\{X^-_t\geq 0\}}=X^-_t\boldsymbol{1}_{\{X^-_t>0\}},\ \text{ as }\vare\uparrow\vare^\ast.\\
\]
Also notice that
\[
0\leq \frac{(t+\vare)^{2H-1}}{X^\vare_t\boldsymbol{1}_{\{X^\vare_t>0\}}+\vare}\leq \frac{(t+\vare^\ast)^{2H-1}}{\vare^\ast},\ \forall \vare>\vare^\ast,\ t\geq0,
\]
and
\[
0\leq \frac{(t+\vare)^{2H-1}}{X^\vare_t\boldsymbol{1}_{\{X^\vare_t>0\}}+\vare}\leq \frac{(t+\vare^\ast/2)^{2H-1}}{\vare^\ast/2},\ \forall \vare\in(\vare^\ast/2,\vare^\ast),\ t\geq0.
\]
Thus, by the dominated convergence theorem we obtain, for any $t>0$,
\begin{equation}\label{plus-2}
\lim\limits_{\vare\downarrow\vare^\ast}\int_0^t\frac{(s+\vare)^{2H-1}}{X^\vare_s\boldsymbol{1}_{\{X^\vare_s>0\}}+\vare}ds=\int_0^t\frac{(s+\vare^\ast)^{2H-1}}{X^+_s\boldsymbol{1}_{\{X^+_s>0\}}+\vare^\ast}ds,
\end{equation}
and
\begin{align}\label{minus-2}
\lim\limits_{\vare \uparrow\vare^\ast}\int_0^t\frac{(s+\vare)^{2H-1}}{X^\vare_s\boldsymbol{1}_{\{X^\vare_s>0\}}+\vare}ds
&=\lim\limits_{\vare \uparrow\vare^\ast}\int_0^t\frac{(s+\vare)^{2H-1}}{X^\vare_s\boldsymbol{1}_{\{X^\vare_s\geq0\}}+\vare}\notag\\
&=\int_0^t\frac{(s+\vare^\ast)^{2H-1}}{X^-_s\boldsymbol{1}_{\{X^-_s\geq0\}}+\vare^\ast}ds\notag\\
&=\int_0^t\frac{(s+\vare^\ast)^{2H-1}}{X^-_s\boldsymbol{1}_{\{X^-_s>0\}}+\vare^\ast}ds.
\end{align}
It implies from  \eqref{plus-1}-\eqref{minus-2} that
\begin{align}\label{plus-3}
X^+_t&=\lim\limits_{\vare\downarrow\vare^\ast}X^\vare_t\notag\\
&=X_0+a\lim\limits_{\vare\downarrow\vare^\ast}\int_0^t\frac{(s+\vare)^{2H-1}}{X^\vare_s\boldsymbol{1}_{\{X^\vare_s>0\}}+\vare}ds-b \lim\limits_{\vare\downarrow\vare^\ast}\int_0^tX^\vare_sds+\si B^H_t \notag\\
&=X_0+a\int_0^t\frac{(s+\vare^\ast)^{2H-1}}{X^+_s\boldsymbol{1}_{\{X^+_s>0\}}+\vare^\ast}ds-b\int_0^tX^+_sds+\si B^H_t,
\end{align}
and 
\begin{align}\label{minus-3}
X^-_t&=\lim\limits_{\vare\uparrow\vare^\ast}X^\vare_t\notag\\
&=X_0+a\lim\limits_{\vare\uparrow\vare^\ast}\int_0^t\frac{(s+\vare)^{2H-1}}{X^\vare_s\boldsymbol{1}_{\{X^\vare_s>0\}}+\vare}ds-b \lim\limits_{\vare\uparrow\vare^\ast}\int_0^tX^\vare_sds+\si B^H_t \notag\\
&=X_0+a\int_0^t\frac{(s+\vare^\ast)^{2H-1}}{X^-_s\boldsymbol{1}_{\{X^-_s>0\}}+\vare^\ast}ds-b\int_0^tX^-_sds+\si B^H_t.
\end{align}
So, $X^+$ and $X^-$ are solutions of \eqref{sde-2} with $\vare=\vare^\ast$. Since $X^{\vare^\ast}$ is the unique solution of \eqref{sde-2} with $\vare=\vare^\ast$, we have $X^{\vare^\ast}=X^+=X^-$, and therefore, we know that $\lim\limits_{\vare\to\vare^\ast}X^\vare_t$ exists and 
\[
\lim\limits_{\vare\to\vare^\ast}X^\vare_t=X^{\vare^\ast}_t.
\]
The proof is completed. \end{proof}
\bigskip

 \begin{lemma} \label{lem-3-6} Let $g(t)$ be a continuous function on $[0,\infty)$ with $g(0)=0$, and for any fixed $T>0$ 
 \[
 |g(t)-g(s)|\leq C|t-s|^\beta,
 \]
 for some constants $C>0$ and $0<\beta<H$. Consider the differential equation 
 \begin{equation}\label{lem-equ}
 x_t=x_0+a\int_0^t \frac{s^{2H-1}}{x_s}ds-b\int_0^tx_sds+g(t),
 \end{equation}
 where $x_0>0$ is a given constant. There exists $\delta>0$ such that equation \eqref{lem-equ} has a unique continuous solution on $[0,\delta]$.
   \end{lemma}
 \begin{proof} Fix $T=1$. Then, there exist $C>0$ and $0<\beta<H$ such that 
 \[
 |g(t)|\leq Ct^\beta, \ \forall t\in [0,1].
 \]
Noticing that the functions $$f(t)=\frac{at^{2H}}{Hx_0}-\frac{bx_0t}{2}+Ct^\beta$$ and $$h(t)=\frac{at^{2H}}{2Hx_0}-bx_0t-Ct^\beta$$
are continuous on $[0,1]$ and that $f(0)=h(0)=0$, we can choose some $0<\delta<1$ such that 
\begin{equation}\label{lem-equ-1}
-\frac{x_0}{2}\leq h(t)\leq f(t)\leq x_0, \ \forall t\in [0,\delta],
\end{equation}
and 
\begin{equation}\label{lem-equ-2}
\frac{2a\delta^{2H}}{Hx_0^2}+b\delta<1.
\end{equation}

 Consider the following complete metric space
 \[
 C_{x_0}([0,\delta])=\left\{y:[0, \delta]\to \left[\frac{x_0}{2},2x_0\right]: y_t \text{ is continuous on } [0,\delta] \text{ and } y_0=x_0\right\}
 \]
 with distance function given by
 \[
 d(y,z)=\Vert y-z\Vert_\infty=\sup\limits_{0\leq t\leq \delta}|y_t-z_t|, \ y, z\in C_{x_0}([0,\delta]).
 \]
Now for each $x\in C_{x_0}([0,\delta])$ define the mapping
\[
(Tx)_t=x_0+a\int_0^t \frac{s^{2H-1}}{x_s}ds-b\int_0^tx_sds+g(t).
\]
Notice that $(Tx)_0=x_0$ and $Tx$ is continuous on $[0,\delta]$ for any  $x\in C_{x_0}([0,\delta])$. It implies from the definition of $C_{x_0}([0,\delta])$ and \eqref{lem-equ-1} that, for any $x\in C_{x_0}([0,\delta])$,
\begin{align*}
(Tx)_t\leq& x_0+\frac{2a}{x_0}\int_0^ts^{2H-1}ds-\frac{bx_0}{2}t+Ct^\beta\\
=& x_0+\frac{at^{2H}}{Hx_0}-\frac{bx_0}{2}t+Ct^\beta\\
=&x_0+f(t)\leq 2x_0, \ \forall t\in [0,\delta],
\end{align*}
and 
\begin{align*}
(Tx)_t\geq& x_0+\frac{a}{x_0}\int_0^ts^{2H-1}ds-bx_0t-Ct^\beta\\
=& x_0+\frac{at^{2H}}{2Hx_0}-bx_0t-Ct^\beta\\
=&x_0+h(t)\geq \frac{x_0}{2}, \ \forall t\in [0,\delta].
\end{align*}

Thus $T$ maps $C_{x_0}([0,\delta])$ into $C_{x_0}([0,\delta])$. Next, let us prove that $T$ is a contraction. Let $x^1, x^2\in C_{x_0}([0,\delta])$. Then, 
\begin{align*}
|(Tx^1)_t-(Tx^2)_t|&\leq a\int_0^ts^{2H-1}\left|\frac{1}{x^1_s}-\frac{1}{x^2_s}\right|ds+b\int_0^t|x^1_s-x^2_s|ds\\
&=a\int_0^ts^{2H-1}\frac{\left|x^1_s-x^2_s\right|}{x_s^1x_s^2}ds+b\int_0^t|x^1_s-x^2_s|ds\\
&\leq \left(\frac{2at^{2H}}{Hx_0^2}+bt\right)\Vert x^1-x^2\Vert_\infty\\
&\leq \left(\frac{2a\delta^{2H}}{Hx_0^2}+b\delta\right)\Vert x^1-x^2\Vert_\infty, \ \forall t\in [0,\delta].
\end{align*}
Thus, we have 
\[
\Vert Tx^1-Tx^2\Vert_\infty\leq \left(\frac{2a\delta^{2H}}{Hx_0^2}+b\delta\right)\Vert x^1-x^2\Vert_\infty.
\] 
Since $\frac{2a\delta^{2H}}{Hx_0^2}+b\delta<1$, the mapping $T$ is a contraction. Thus, $T$ has a unique fixed point $x$ in $C_{x_0}([0,\delta])$, and hence, $x$ is the unique solution of \eqref{lem-equ}  on $[0,\delta]$.
 \end{proof}

 \begin{theorem}\label{thm-3-7}
 The limit process $X=\{X_t,t\geq 0\}$ has the following properties:
 \begin{itemize}
 \item[(a)] the set $\{t>0: X_t>0\}$ is open in the natural topology on $\RR$;
 \item[(b)] $X$ is continuous on $\{t\geq 0: X_t>0\}$.
 \end{itemize}
 \end{theorem}
\begin{proof}
 (a). Consider arbitrary $t^\ast \in \{t>0: X_{t}>0\}.$ Since 
 \[
 X^{\vare}_t\uparrow  X_{t}, \ \forall t\geq 0, \ \text{ as } \vare\downarrow 0,
 \]
 There is some $\vare_0>0$ such that $X^{\vare_0}_{t^\ast}>0$. 
 Since $X^{\vare_0}_t$ is continuous in $t$, there exists $\delta_0>0$ such that $[t^\ast-\delta_0, t^\ast+\delta_0]\subset (0,\infty)$ and 
 \[
 X_t^{\vare_0}>0, \ \forall t\in [t^\ast-\delta_0, t^\ast+\delta_0].
 \]
 Then, we obtain 
 \[
 X_t\geq X_t^{\vare_0}>0, \ \forall t\in [t^\ast-\delta_0, t^\ast+\delta_0].
 \] 
 So, $(t^\ast-\delta_0, t^\ast+\delta_0)\subset [t^\ast-\delta_0, t^\ast+\delta_0]\subseteq \{t>0: X_{t}>0\}$.
 Therefore, $\{t>0: X_{t}>0\}$ is open.\\

(b). Let us first prove the continuity of $X$ at $t^\ast \in \{t>0: X_{t}>0\}$. 

Since $X^{\vare}\uparrow  X$ as $\vare\downarrow 0$, it implies from the proof of Part (a) that there exists $\vare_0$ and $\delta_0$ such that
\begin{equation}\label{thm-3-7-1}
X_t\geq X_t^\vare\geq X_t^{\vare_0}>0, \ \forall t\in [t^\ast-\delta_0, t^\ast+\delta_0], \ 0<\vare\leq \vare_0. 
\end{equation}
We shall prove that $\{X^\vare\}$ converges to $X$ uniformly on $[t^\ast-\delta_0, t^\ast+\delta_0]$ as $\vare\to 0$. Since each $X^\vare$ is continuous on $[t^\ast-\delta_0, t^\ast+\delta_0]$, the uniform convergence of $\{X^\vare\}$ to $X$ on the compact set $[t^\ast-\delta_0, t^\ast+\delta_0]$ implies the continuity of $X$ on $[t^\ast-\delta_0, t^\ast+\delta_0]$. In particular, $X$ is continuous at $t^\ast$.

Applying Theorem \ref{thm-2-2} with $T=t^\ast+\de_0$, we can see that, for $0<\vare\leq \vare_0$, 
\begin{equation}\label{thm-3-7-2}
0\leq \sup_{t^\ast-\de_0\leq t\leq t^\ast+\de_0} \{X^\vare_t-X^{\vare_0}_t\}\leq 2C+4\sigma\sup_{0\leq s\leq t^\ast+\de_0}|B_s^H|,
\end{equation}
namely, $\{X^\vare-X^{\vare_0}\}_{0<\vare\leq \vare_0}$ is uniformly bounded on $[t^\ast-\delta_0, t^\ast+\delta_0]$.

For $0<\vare\leq \vare_0$, by \eqref{thm-3-7-1} 
we get $\mathbf{1}_{\{X^{\vare}_{s} > 0\}}=\mathbf{1}_{\{X^{\vare_0}_{s} > 0\}}=1$ for all $s \in[t^*-\de_0,t^*+\de_0]$. Then, it implies from \eqref{sde-2} that $X^\vare-X^{\vare_0}$ satisfies the following differential equation on  $[t^\ast-\delta_0, t^\ast+\delta_0]$

\begin{align*}
&X^\vare_t-X^{\vare_0}_t\\
=&a\int_{t^\ast-\de_0}^t\left(\frac{(s+\vare)^{2H-1}}{X_s^{\vare}\mathbf{1}_{\{X^{\vare}_{s} > 0\}}+\vare}-\frac{(s+\vare_0)^{2H-1}}{X_s^{\vare_0}\mathbf{1}_{\{X^{\vare_0}_{s} > 0\}}+\vare_0}\right)ds-b\int_{t^\ast-\de_0}^t(X_s^{\vare}-X_s^{\vare_0})ds\\
=&a\int_{t^\ast-\de_0}^t\left(\frac{(s+\vare)^{2H-1}}{X_s^{\vare}+\vare}-\frac{(s+\vare_0)^{2H-1}}{X_s^{\vare_0}+\vare_0}\right)ds-b\int_{t^\ast-\de_0}^t(X_s^{\vare}-X_s^{\vare_0})ds.
\end{align*}

By \eqref{thm-3-7-1} and the continuity of of $X^\vare_t$ and $X^{\vare_0}_t$ in $t\in[t^\ast-\delta_0, t^\ast+\delta_0]$, $X^\vare-X^{\vare_0}$ is continuously differentiable on  $[t^\ast-\delta_0, t^\ast+\delta_0]$ and 
\begin{align*}
    (X^{\vare}_t-X^{\vare_0}_t)'=& a\left(\frac{(t+\vare)^{2H-1}}{X^\vare_{t}+\vare}-\frac{(t+\vare_0)^{2H-1}}{X^{\vare_0}_{t}+\vare_0}\right)-b(X^{\vare}_{t}-X^{\vare_0}_{t}). 
    \end{align*}
Then, for all $t\in [t^\ast-\delta_0, t^\ast+\delta_0]$ and all $0<\vare\leq \vare_0$, we can obtain the following estimate by \eqref{thm-3-7-2}
\begin{align*}
&\left|(X^{\vare}_t-X^{\vare_0}_t)'\right|\\
\leq& \frac{a(t+\vare)^{2H-1}}{X^\vare_{t}+\vare}+\frac{a(t+\vare_0)^{2H-1}}{X^{\vare_0}_{t}+\vare_0}+b (X^{\vare}_{t}-X^{\vare_0}_{t})\\
\leq& \frac{a(t^\ast-\de_0)^{2H-1}}{X^{\vare_0}_{t}}+\frac{a(t^\ast-\de_0+\vare_0)^{2H-1}}{X^{\vare_0}_{t}+\vare_0}+b\left(2C+4\sigma\sup_{0\leq s\leq t^\ast+\de_0}|B_s^H|\right)\\
\leq & \frac{a(t^\ast-\de_0)^{2H-1}}{m(\vare_0)}+\frac{a(t^\ast-\de_0+\vare_0)^{2H-1}}{m(\vare_0)+\vare_0}+b\left(2C+4\sigma\sup_{0\leq s\leq t^\ast+\de_0}|B_s^H|\right),
\end{align*}
where $m(\vare_0)=\inf\limits_{t^\ast-\de_0\leq t\leq t^\ast+\de_0} X^{\vare_0}_t>0$.
Thus, $\{X^\vare-X^{\vare_0}\}_{0<\vare\leq \vare_0}$ is equicontinuous on $[t^\ast-\delta_0, t^\ast+\delta_0]$. 

Together  with the monotonicity of $\{X^\vare-X^{\vare_0}\}_{0<\vare\leq \vare_0}$ in $\vare$, the pointwise convergence of $\{X^\vare-X^{\vare_0}\}_{0<\vare\leq \vare_0}$ to $X-X^{\vare_0}$, and the uniform boundedness of  $\{X^\vare-X^{\vare_0}\}_{0<\vare\leq \vare_0}$, we obtain that
$\{X^\vare-X^{\vare_0}\}_{0<\vare\leq \vare_0}$  converges  to $X-X^{\vare_0}$ uniformly on $[t^\ast-\delta_0, t^\ast+\delta_0]$, which implies that $\{X^\vare\}_{0<\vare\leq \vare_0}$ converges  to $X$ uniformly on $[t^\ast-\delta_0, t^\ast+\delta_0]$. 

Next, let us prove the continuity of $X$ at $t=0\in \{t\geq 0: X_t>0\}$. With $g(t)=B^H_t$ and $x_0=X_0$ in Lemma \ref{lem-3-6}, we can see that $X_t$ is the unique continuous solution of \eqref{lem-equ} on some interval $[0,\delta]$ for some $\delta>0$. Thus, $X_t$ is continuous at $t=0$.
\end{proof}
\begin{corollary}\label{coro-3-8}
The trajectories of $X=\{X_t, t\geq 0\}$ are continuous almost everywhere on $[0,\infty)$ and therefore are Riemann integrable on any bounded interval $[a,b]\subset[0,\infty)$.
\end{corollary}
\begin{proof}
The result follows immediately from Corollary \ref{coro-1} and Theorem \ref{thm-3-7}.
\end{proof}

Since the set $\{t>0: X_t>0\}$ is open in the natural topology of $\mathbb{R}$, it can be written as at most countable union of disjoint open intervals, that is, there exists $N\in \mathbb{N}\bigcup \{\infty\}$ such that 
\[
\{t>0: X_t>0\}=\bigcup_{i=0}^N(\alpha_i, \beta_i),
\]
where $\alpha_0=0$ and $(\alpha_i,\beta_i)\bigcap (\alpha_j,\beta_j)=\emptyset$ for $i\neq j$. Then, the set $\{t\geq0: X_t>0\} $ can be written as
\begin{equation}\label{interval}
\{t\geq0: X_t>0\}=[\alpha_0,\beta_0)\cup\left(\bigcup_{i=1}^N(\alpha_i, \beta_i)\right).
\end{equation}

Note from Corollary \ref{coro-1} that $X_{\alpha_i}=0$ for all $i\geq 1$ and $X_{\beta_i}=0$ for all $i\geq 0$.\\

\begin{theorem}\label{thm-3-9}
Let $(\alpha_i,\beta_i)$, $i\geq 1$, be an interval in the presentation of \eqref{interval}. Then,
\begin{itemize}
\item[(a)] $\lim\limits_{t\downarrow \alpha_i}X_t=0=X_{\alpha_i}$ and $\lim\limits_{t\uparrow \beta_i}X_t=0=X_{\beta_i}$;
\item[(b)] for any $t\in [\alpha_i,\beta_i]$, we have
\[
X_t=a\int_{\alpha_i}^t\frac{s^{2H-1}}{X_s}ds-b\int_{\alpha_i}^t X_sds+B^H_t-B^H_{\alpha_i}.
\]
\end{itemize}
\end{theorem}
\begin{proof}
(a). Let us first prove the limit at the left endpoint of the interval.

Since $X_t>0$ on $(\alpha_i,\beta_i)$, we only need to prove that $\limsup\limits_{t\downarrow \alpha_i}X_t=0$.

The boundedness of $X$ on $(\alpha_i,\beta_i)$ in \eqref{limit-bound} implies that $x:=\limsup\limits_{t\downarrow \alpha_i}X_t<\infty$. Assume that $x>0$. By the definition of limsup, for any $\delta>0$, there exists $\tau\in (\alpha_i, \alpha_i+\delta)$ such that $X_\tau\in \left(\frac{3x}{4}, \frac{5x}{4}\right)$.

Since $X^\vare_\tau\uparrow X_\tau$, there exists $\vare(\tau)$ such that $X^\vare_\tau\in \left(\frac{x}{2},\frac{5x}{4}\right)$ for all $0<\vare\leq \vare(\tau)$.

Noting that $X^{\vare_1}_t>X^{\vare_2}_t$ for all $t$ whenever $0<\vare_1<\vare_2$ and that $X_{\alpha_i}=0$, we can easily see that $X^\vare_{\alpha_i}<0$ for all $\vare>0$.

Then, for $0<\vare\leq \vare(\tau)$, the continuity of $X^\vare_t$ in $t$ implies that there exists $\tau_1\in (\alpha_i, \tau)$ such that 
\[
\tau_1=\sup\left\{t\in (\alpha_i, \tau): X^\vare_t=\frac{x}{4}\right\}
\]
and $X^\vare_{\tau_1}=\frac{x}{4}$. So, $X^\vare_\tau-X^\vare_{\tau_1}>\frac{x}{4}$ and $X^\vare_t>\frac{x}{4}>0$ for all $t\in (\tau_1, \tau]$. 

Now applying \eqref{fbm-hold} with large enough $T$, $\gamma=\frac{H}{2}$,  \eqref{sde-2} and the inequality $|u^{2H}-v^{2H}|\leq |u-v|^{2H}$ for all $u, v\geq 0$,  one can obtain
\begin{align*}
0<\frac{x}{4}<X^\vare_\tau-X^\vare_{\tau_1}&=a\int_{\tau_1}^\tau\frac{(s+\vare)^{2H-1}}{X_s^{\vare}\mathbf{1}_{\{X^{\vare}_{s} > 0\}}+\vare}ds-b\int_{\tau_1}^\tau X_s^{\vare}ds+\sigma(B^H_\tau-B^H_{\tau_1})\\
&\leq\frac{2a}{Hx}\left(\tau^{2H}-\tau_1^{2H}\right)+C\sigma (\tau-\tau_1)^{\frac{H}{2}}\\
&\leq \frac{2a}{Hx}\left(\tau-\tau_1\right)^{2H}+C\sigma (\tau-\tau_1)^{\frac{H}{2}}\\
&\leq \frac{2a}{Hx}\delta^{2H}+C\sigma \delta^{\frac{H}{2}},
\end{align*}
which is not true for arbitrary small $\delta$. Hence,
\[
0\leq \liminf\limits_{t\downarrow \alpha_i}X_t\leq \limsup\limits_{t\downarrow \alpha_i}X_t=0.
\]
Thus, $\lim\limits_{t\downarrow \alpha_i}X_t=0$. By analogous analysis, we can obtain $\lim\limits_{t\uparrow \beta_i}X_t=0$.

(b) For any bounded closed interval $[\alpha_i^\ast, \beta_i^\ast]\subset (\alpha_i,\beta_i)$, Theorem \ref{thm-3-7} shows that $X$ is continuous on $[\alpha_i^\ast, \beta_i^\ast]$. Since the family $\{X^\vare\}$ of continuous functions monotonically increases to $X$ on $[\alpha_i^\ast, \beta_i^\ast]$, Dini's theorem implies that $X^\vare\to X$ uniformly on $[\alpha_i^\ast, \beta_i^\ast]$ as $\vare\to 0$.

Since $X$ is continuous on $[\alpha_i^\ast, \beta_i^\ast]$ and $X_t>0$ for all $t\in [\alpha_i^\ast, \beta_i^\ast]$, we have 
$$0<m:=\inf\limits_{t\in [\alpha_i^\ast, \beta_i^\ast]}X_t\leq \sup\limits_{t\in [\alpha_i^\ast, \beta_i^\ast]}X_t=:M<\infty.$$
Since $\{X^\vare\}$ converges uniformly to $X$ on $[\alpha_i^\ast, \beta_i^\ast]$, there exists $\vare_0>0$ such that 
\[
X^\vare_t\geq\frac{m}{2}, \ \forall t\in [\alpha_i^\ast, \beta_i^\ast], \vare\in (0,\vare_0].
\]
Hence, for all $t\in [\alpha_i^\ast, \beta_i^\ast]$ and all $\vare\in (0,\vare_0]$, we have
\begin{align*}
&\left|\frac{(t+\vare)^{2H-1}}{X^\vare_t\mathbf{1}_{\{X^{\vare}_{t} > 0\}}+\vare}-\frac{t^{2H-1}}{X_t\mathbf{1}_{\{X_{t} > 0\}}}\right|\\
=&\left|\frac{(t+\vare)^{2H-1}}{X^\vare_t+\vare}-\frac{t^{2H-1}}{X^\vare_t}\right|\\
=&\left|\frac{(t+\vare)^{2H-1}X_t-t^{2H-1}(X^\vare_t+\vare)}{(X^\vare_t+\vare)X_t}\right|\\
\leq&\frac{2}{m^2}\left(\left|((t+\vare)^{2H-1}-t^{2H-1})X_t\right|+t^{2H-1}\left|X_t^\vare+\vare-X_t\right|\right)\\
\leq &\frac{2}{m^2}\left(M(1-2H)(\alpha_i^\ast)^{2H-2}\vare+(\alpha_i^\ast)^{2H-1}(|X^\vare_t-X_t|+\vare)\right).
\end{align*}
Thus, $\left\{\frac{(t+\vare)^{2H-1}}{X^\vare_t\mathbf{1}_{\{X^{\vare}_{t} > 0\}}+\vare}=\frac{(t+\vare)^{2H-1}}{X^\vare_t+\vare}\right\}_{0<\vare\leq \vare_0}$ converges uniformly to $\frac{t^{2H-1}}{X_t\mathbf{1}_{\{X_{t} > 0\}}}=\frac{t^{2H-1}}{X_t}$ on $[\alpha_i^\ast, \beta_i^\ast]$. Together with the uniform convergence to $\{X^\vare\}$ to $X$, we obtain,  for $t\in [\alpha_i^\ast, \beta_i^\ast]$,
\begin{align*}
a\int_{\alpha_i^\ast}^t\frac{s^{2H-1}}{X_s}ds&=\lim\limits_{\vare\to 0}a\int_{\alpha_i^\ast}^t\frac{(s+\vare)^{2H-1}}{X^\vare_s\mathbf{1}_{\{X^{\vare}_{s} > 0\}}+\vare}ds\\
&=\lim\limits_{\vare\to 0}\left(X^\vare_t-X^\vare_{\alpha_i^\ast}+b\int_{\alpha_i^\ast}^tX_s^\vare ds-\sigma(B^H_t-B^H_{\alpha_i^\ast})\right)\\
&=X_t-X_{\alpha_i^\ast}+b\int_{\alpha_i^\ast}^tX_s ds-\sigma(B^H_t-B^H_{\alpha_i^\ast}),
\end{align*}
or equivalently, 
\begin{equation}\label{thm-3-9-1}
X_t=X_{\alpha_i^\ast}+a\int_{\alpha_i^\ast}^t\frac{s^{2H-1}}{X_s}ds-b\int_{\alpha_i^\ast}^tX_s^\vare ds+\sigma(B^H_t-B^H_{\alpha_i^\ast}).
\end{equation}
Since $X$ is Riemann integrable on any bounded interval (see Corollary \ref{coro-3-8}), $\frac{s^{2H-1}}{X_s}$ is Lebesgue integrable on any bounded interval (see Corollary \ref{coro-2}), and $X$ and $B$ are right continuous at $\alpha_i$ (see Theorem \ref{thm-3-9} for the right continuity of $X$ at $\alpha_i$), we  obtain
\[
\lim\limits_{\alpha_i^\ast\downarrow \alpha_i}X_{\alpha_i^\ast}=X_{\alpha_i}=0,
\] 
\[
\lim\limits_{\alpha_i^\ast\downarrow \alpha_i}\int_{\alpha_i^\ast}^t\frac{s^{2H-1}}{X_s} ds=\int_{\alpha_i}^t\frac{s^{2H-1}}{X_s} ds,
\]
\[
\lim\limits_{\alpha_i^\ast\downarrow \alpha_i}\int_{\alpha_i^\ast}^tX_sds=\int_{\alpha_i}^tX_sds,
\]
and 
\[
\lim\limits_{\alpha_i^\ast\downarrow \alpha_i}B^H_{\alpha_i^\ast}=B^H_{\alpha_i}.
\]
Sending $\alpha_i^\ast$ to $\alpha_i$ in \eqref{thm-3-9-1}, we get, for $t\in [\alpha_i, \beta_i)$,
\[
X_t=a\int_{\alpha_i}^t\frac{s^{2H-1}}{X_s}ds-b\int_{\alpha_i}^tX_s^\vare ds+\sigma(B^H_t-B^H_{\alpha_i})
\]
By taking $t\uparrow \beta_i$ in the above equation we can get the result  at $t=\beta_i$.
\end{proof}
\begin{remark}
Similar to the proof of Theorem \ref{thm-3-9}, we can easily prove that
\[
\lim\limits_{t\uparrow \beta_0}X_t=0=X_{\beta_0},
\]
and, for $t\in [0,\beta_0]$,
\[
X_t=X_0+a\int_0^t\frac{s^{2H-1}}{X_s}ds-b\int_0^tX_s^\vare ds+\sigma B^H_t.
\]
\end{remark}

\bibliographystyle{plain}
\bibliography{ref}
\end{document}